\documentclass[]{amsart}
\usepackage{amsmath, amsfonts, amssymb, amsthm}

\usepackage[hidelinks]{hyperref}

\newcommand{\FF}{\mathbb{F}}
\newcommand{\CC}{\mathbb{C}}
\newcommand{\afrak}{\mathfrak{a}}
\newcommand{\pfrak}{\mathfrak{p}}
\newcommand{\qfrak}{\mathfrak{q}}
\newcommand{\mfrak}{\mathfrak{m}}
\newcommand{\nfrak}{\mathfrak{n}}
\newcommand{\pitilde}{\widetilde{\pi}}

\DeclareMathOperator{\GL}{GL}

\newtheorem{theorem}{Theorem}[section]
\newtheorem{proposition}[theorem]{Proposition}
\newtheorem{lemma}[theorem]{Lemma}
\newtheorem{corollary}[theorem]{Corollary}

\theoremstyle{remark}
\newtheorem{remark}[theorem]{Remark}

\theoremstyle{definition}
\newtheorem{definition}[theorem]{Definition}
\numberwithin{equation}{section}

\title{Twisting eigensystems of Drinfeld Hecke eigenforms by characters}
\author{R. Perkins}
\email{rudolph.perkins@iwr.uni-heidelberg.de}
\thanks{The author is supported by the Alexander von Humboldt Foundation.}
\address{IWR, University of Heidelberg, Im Neuenheimer Feld 205, 69120 Heidelberg, Germany}
\date{\today}

\begin{document}

\begin{abstract}
We address some questions posed by Goss related to the modularity of Drinfeld modules of rank 1 defined over the field of rational functions in one variable with coefficients in a finite field. 

For each positive characteristic valued Dirichlet character, we introduce certain projection operators on spaces of Drinfeld modular forms with character of a given weight and type such that when applied to a Hecke eigenform return a Hecke eigenform whose eigensystem has been twisted by the given Dirichlet character. 
Unlike the classical case, however, the effect on Goss' $u$-expansions for these eigenforms --- and even on Petrov's $A$-expansions --- is more complicated than a simple twisting of the $u-$ (or $A-$) expansion coefficients by the given character.

We also introduce Eisenstein series with character for irreducible levels $\pfrak$ and show that they and their Fricke transforms are Hecke eigenforms with a new type of $A$-expansion and $A$-expansion in the sense of Petrov, respectively. We prove congruences between certain cuspforms in Petrov's special family and the Eisenstein series and their Fricke transforms introduced here, and we show that in each weight there are as many linearly independent Eisenstein series with character as there are cusps for $\Gamma_1(\pfrak)$.
\end{abstract}

\maketitle

\section{Introduction}
\subsection{Set-up}
Let $\FF_q$ be the finite field with $q$ elements of characteristic $p$. Let $A := \FF_q[\theta]$, $K := \FF_q(\theta)$, and $K_\infty := \FF_q((1/\theta))$, the completion of $K$ with respect to the non-archimedean absolute value $|\cdot|$ normalized so that $|\theta| = q$. Let $\CC_{\infty}$ be the completion of an algebraic closure of $K_\infty$ equipped with the canonical extension of $|\cdot|$, also denoted by the same symbol. Finally, let $A_+$ denote the multiplicative submonoid of $A$ consisting of the monic polynomials in $\theta$. 

Let $\Omega := \CC_\infty \setminus K_\infty$ be Drinfeld's period domain of rank two Drinfeld modules equipped with its usual structure as a rigid analytic space. For $z \in \CC_\infty$, let 
\[|z|_\Im := \inf_{\kappa \in K_\infty} |z - \kappa|;\] 
trivially, for all $z \in \Omega$, $|z| \geq |z|_\Im$. 
The group $\GL_2(K_\infty)$ acts on $\Omega$ via linear fractional transformations compatibly with the rigid analytic structure; i.e. for $\gamma = \left( \begin{smallmatrix} a & b \\ c & d  \end{smallmatrix} \right) \in \GL_2(K_\infty)$ and $z \in \Omega$, we define $\gamma z := \frac{az+b}{cz+d}$. 

We let $\gamma \in \GL_2(K)$ act on rigid analytic functions $f : \Omega \rightarrow \CC_\infty$ via 
\[f|_k^m[\gamma] : z \mapsto (\det \gamma)^m j(\gamma,z)^{-k} f(\gamma z),\] where, as usual, $j(\left( \begin{smallmatrix} a & b \\ c & d  \end{smallmatrix} \right),z) = cz+d$ and $m,k$ can be arbitrary non-negative integers. 
We recall the formula for the composition of two slash operators which we employ tacitly throughout the sequel: \medskip

{\it For all $k,m_1,m_2 \geq 0$ and $\gamma_1, \gamma_2 \in \GL_2(K)$,}
\[ (\cdot|_k^{m_1}[\gamma_1])|_k^{m_2}[\gamma_2] = (\det\gamma_2)^{m_2 - m_1} (\cdot|_k^{m_1}[\gamma_1\gamma_2]) . \]
For $m = 0$, the formula is the same, but we shall often just write $f|_k[\gamma]$ for $f|_k^0[\gamma]$. Similarly, if $m=k=0$, we may write $f|[\gamma] := f|_0^0[\gamma]$.

\subsubsection{Drinfeld Modular Forms} 
Throughout the sequel, $\pfrak$ will denote a monic irreducible polynomial of $A$, and $\mfrak$ an arbitrary monic polynomial.

For each such $\mfrak \neq 1$, we have the following three fundamental subgroups of $\Gamma(1) := \GL_2(A)$:
\begin{eqnarray*}
\Gamma(\mfrak) &:=& \left\{ \gamma \in \GL_2(A) : \gamma \equiv \left( \begin{smallmatrix} 1 & 0 \\ 0 & 1  \end{smallmatrix} \right) \mod{\mfrak A} \right\}, \\
\Gamma_0(\mfrak) &:=& \left\{ \left( \begin{smallmatrix} a & b \\ c & d  \end{smallmatrix} \right) \in \GL_2(A) : c \in \mfrak A \right\}, \\
\Gamma_1(\mfrak) &:=& \left\{ \left( \begin{smallmatrix} a & b \\ c & d  \end{smallmatrix} \right) \in \GL_2(A) : c \in \mfrak A  \text{ and } d \equiv 1 \pmod{\mfrak A} \right\},
\end{eqnarray*}
and we wish to study spaces of $\CC_\infty$-valued modular forms for the latter two groups. 
While such forms were introduced by D. Goss in his thesis, they are now commonly dubbed Drinfeld modular forms.

\begin{definition}(\cite[(4.1)]{EGjnt01})
Let $\Gamma$ be a subgroup of $\Gamma(1)$ containing $\Gamma(\mfrak)$ for some $\mfrak$. The $\mfrak$ of least degree will be called the \emph{level} of $\Gamma$. 

We write $M_k^m(\Gamma)$ for those rigid analytic functions $f: \Omega \rightarrow \CC_\infty$ such that 
\begin{eqnarray} f|_k^m[\gamma] = f, & \quad \forall \gamma \in \Gamma, \text{ and} \\
 f|_k^m[\gamma] \text{ is bounded on } \{z \in \Omega : |z|_\Im \geq 1\}, & \quad \forall \gamma \in \Gamma(1). 
\end{eqnarray}

We say such $f \in M_k^m(\Gamma)$ is a  {\it Drinfeld modular form} for $\Gamma$ of {\it weight} $k$ and {\it type} $m$. 

If, for all $\gamma \in \Gamma(1)$, we have $f|_k^m[\gamma](z) \rightarrow 0$ as $|z|_\Im \rightarrow \infty$, we say $f$ is a {\it cuspidal Drinfeld modular form} and write $f \in S_k^m(\Gamma)$. 

As for the slash operators above we may abbreviate $M_k^0(\Gamma)$ by $M_k(\Gamma)$.
\end{definition}

\subsubsection{Expansions at the cusps}
By Goss' Lemma, e.g. \cite[Theorem 4.2]{DGbams}, any rigid analytic function on $\Omega$ invariant under $z \mapsto z+a$, for all $a \in \mfrak A$ (i.e. which is $\mfrak A$\emph{-periodic}), has a doubly infinite series expansion in the parameter
\[u_\mfrak(z) := \frac{\mfrak}{\pitilde}\sum_{a \in \mfrak A} \frac{1}{z+a}, \]
which converges for $|z|_\Im$ sufficiently large; here $\pitilde$ is a fixed choice of fundamental period of the Carlitz module.
Notice that 
\[u_\mfrak(z) \rightarrow 0 \text{ as } |z|_\Im \rightarrow \infty, \text{ and that}\]
\[ u_\mfrak(z) = u_1(z/\mfrak) := \frac{1}{\pitilde} \sum_{a \in A} \frac{1}{z/\mfrak + a}. \]
Thus, any function which is bounded on  $\{z \in \Omega : |z|_\Im \geq 1\}$ and $\mfrak A$-periodic will have a power series expansion in $\CC_\infty[[u_\mfrak]]$. NB. this power series uniquely determines the function as $\Omega$ is connected as a rigid analytic space.

We recall the Carlitz exponential function defined by 
\begin{equation} \label{expCeq}\exp_C(z) := z\prod_{0 \neq a \in A} (1 - \frac{z}{\pitilde a}),\end{equation}
where $\pitilde$ is the fundamental period of the Carlitz module fixed above, algebraic over $K_\infty$. 
The function $\exp_C$ is entire and $\FF_q$-linear in $z$ with coefficients in $K$ --- a result originally due to Carlitz. 
One has the basic relation 
\begin{equation} \label{expueq} u_1(z) = \frac{1}{\exp_C(\pitilde z)},
\end{equation}
an identity of rigid meromorphic functions on $\CC_\infty \setminus A$, and we will write $u$ for $u_1$ to follow. 

\subsubsection{Single, double, and $i$-times cuspidal} 
Fix a congruence subgroup $\Gamma \subset \Gamma(1)$. For each $\gamma \in \Gamma(1)$, there is a monic element $\mfrak_\gamma \in A$ of least degree such that for all $f \in M_k^m(\Gamma)$ the form $f|_k^m[\gamma]$ is $\mfrak_\gamma A$-periodic, and hence $f|_k^m[\gamma]$ has an expansion in the parameter $u_{\mfrak_\gamma}$.

We say that a modular form $f \in M_k^m(\Gamma)$ is {\it $i$-times cuspidal} (here $\Gamma$ is fixed), if for all $\gamma \in \Gamma(1)$ the $u_{\mfrak_\gamma}$-expansion for $f|_k^m[\gamma]$ is divisible by $u_{\mfrak_\gamma}^i$.  In particular, we say that such an $f$ is {\it exactly $i$-times cuspidal} for $\Gamma$ if $f$ is $i$-times cuspidal and if for some $\gamma \in \Gamma(1)$, $f|_k^m[\gamma]$ vanishes exactly to the order $i$ in $u_{\mfrak_\gamma}$. When $i = 1$ (resp. $2$) we say that $f$ is single (resp. double) cuspidal.

\subsubsection{Hecke operators}
Throughout $\qfrak$ will denote a monic irreducible polynomial in $A$, distinct from $\pfrak$.
The following lemma will allow us to define Hecke operators on the spaces $M_k^m(\Gamma_1(\mfrak))$. The proof is elementary, and we omit it for the sake of brevity; see \cite[pp. 104--105]{DSbook} for a plan of proof in the classical case.

\begin{lemma} \label{repslem}
Let $\mfrak, \qfrak \in A_+$, with $\qfrak$ additionally irreducible. 

If $\qfrak \not| \mfrak$, then the matrices $\left( \begin{smallmatrix} 1 & \beta \\ 0 & \qfrak \end{smallmatrix} \right)$, with $|\beta|< |\qfrak|$, and any matrix $\left( \begin{smallmatrix} \mu & \nu \\ \mfrak & \qfrak \end{smallmatrix} \right)\left( \begin{smallmatrix} \qfrak & 0 \\ 0 & 1 \end{smallmatrix} \right) \in Mat_2(A)$ such that $\left( \begin{smallmatrix} \mu & \nu \\ \mfrak & \qfrak \end{smallmatrix} \right) \in \operatorname{SL}_2(A)$ give a full set of distinct representatives for the quotient ${\Gamma_i(\mfrak)}\backslash{\Gamma_i(\mfrak)\left( \begin{smallmatrix} 1 & 0 \\ 0 & \qfrak \end{smallmatrix} \right)\Gamma_i(\mfrak)}$, for both $i = 0,1$.

If $\qfrak | \mfrak$, the matrices $\left( \begin{smallmatrix} 1 & \beta \\ 0 & \qfrak \end{smallmatrix} \right)$, with $|\beta|< |\qfrak|$ give a full set of distinct representatives. \hfill \qed
\end{lemma}

\begin{definition} \label{Heckeactiondef}
For $f \in M_k^m(\Gamma_1(\mfrak))$ and a monic irreducible $\qfrak \in A$, we define
\[T_\qfrak f := \left\{ \begin{array}{ll} \qfrak^{k-m}\left(\sum_{|\beta|<|\qfrak|} f|_k^m[\left( \begin{smallmatrix} 1 & \beta \\ 0 & \qfrak \end{smallmatrix} \right)] + f|_k^m[\left( \begin{smallmatrix} \mu\qfrak & \nu \\ \mfrak\qfrak & \qfrak \end{smallmatrix} \right)] \right), & \qfrak \not| \mfrak \\ \qfrak^{k-m}\sum_{|\beta|<|\qfrak|} f|_k^m[\left( \begin{smallmatrix} 1 & \beta \\ 0 & \qfrak \end{smallmatrix} \right)], & \qfrak | \mfrak, \end{array} \right. \]
where for $(\mfrak,\qfrak) = 1$ we take any $\mu,\nu \in A$ such that $\mu\qfrak - \nu\mfrak = 1$.
\end{definition}

As usual, for each irreducible $\qfrak$, the Hecke operator $T_\qfrak$ acts on $M_k^m(\Gamma_1(\mfrak))$ preserving both single cuspidality and double cuspidality.

\subsubsection{Some basic examples} \label{MFexamples}
1. An important first example is the family of Eisenstein series for $\Gamma(\mfrak)$ whose definition goes back to Goss \cite{DGcompo}. For 
\begin{equation} \label{Vmdefeq}
v \in V_\mfrak := (A/\mfrak A)^2 \setminus (0,0), 
\end{equation} 
let 
\[E^{(k)}_v(z) := \sum_{\substack{a,b \in A \\ (a,b) \equiv v (\mfrak A)}} \frac{1}{(az+b)^k}. \]
To see that these forms are bounded at the cusps, notice that each individual $E_v^{(k)}$ is bounded at infinity and observe that $\Gamma(1)$ acts on the right of $V_\mfrak$ in the natural way so that we have
\[E^{(k)}_v |_k [\gamma] = E^{(k)}_{v\gamma}, \text{ for all } \gamma \in \Gamma(1).\]
Cornelissen has shown that a subset of these forms given by fixing the weight $k$ and varying $v$ in any subset of $V_\mfrak$ in bijection with the cusps of $\Gamma(\mfrak)$, gives a basis for the non-cuspidal forms for $\Gamma(\mfrak)$ of weight $k$, \cite[(1.12)]{GCjnt}. Finally, we mention that Gekeler has determined the precise location of the zeros of $E_v^{k}$ in \cite{EGjnt12}. 

2. As another example, A. Petrov has shown\footnote{We make a brief historical note. Two of the forms in Petrov's special family were known prior to his investigations. 
When $s = 0$, one obtains the expansion at infinity for the false Eisenstein series $E$ of Gekeler (see \cite{EGinv} where this function was first discovered), a Drinfeld quasi-modular form of weight $2$, type $1$ and depth $1$ (see \cite{VBFPqmf} for these notions); we shall discuss this form more below.
When $s = 1$, this is the $A$-expansion obtained by L\'opez \cite{BLadm10} for a suitable multiple of the unique normalized cusp form $h$ of weight $q+1$ and type $1$ for $\Gamma(1)$. } in \cite{APjnt} that for all positive integers $s$, the function
\begin{equation}\label{specialfam}
f_s(z) := \sum_{a \in A_+} a^{1+s(q-1)} u(az)
\end{equation}
represents a cusp form of weight $2+s(q-1)$ and type $1$ for $\Gamma(1)$. 
Such an expansion is an example of what we will call an {\it $A$-expansion in the sense of Petrov}. The remaining examples discovered by Petrov arise from the family $f_s$ through the action of hyperderivatives in the $z$ variable, as shown in \cite{APjnt2}. For example, with this approach one may obtain the following $A$-expansion, due to L\'opez \cite{BLadm10}, for the cuspidal Hecke eigenform
\[ \Delta(z) = \sum_{a \in A_+} a^{q(q-1)} u(az)^{q-1}. \]

The interest in $A$-expansions in general lies in their good properties with respect to the Hecke operators, which are also indexed by the monic elements of $A$. In particular, Petrov's family gives canonical representatives for the space of single cusp forms modulo double cusp forms for $\Gamma(1)$ which are themselves Hecke eigenforms; indeed, for each $s \geq 1$ and each irreducible $\qfrak$ one has
\begin{equation} \label{fseigenvals}
 T_\qfrak f_s = \qfrak f_s. 
 \end{equation} 
Additionally, they give a completely explicit expansion near the infinite cusp on $\Omega$ in contrast to the $u$-expansions guaranteed by Goss whose coefficients are very difficult to compute in general; see e.g. \cite[\S 10]{EGinv} where some of these coefficients are computed for forms for $\Gamma(1)$ of small weights.

3. Finally, there is a canonical choice of representative for the space of single cuspidal forms modulo double cuspidal forms of weight two and type 1 for $\Gamma_0(\pfrak)$; see also \cite{Pel-PerkVMF}. We recall the false Eisenstein series of Gekeler normalized via its $A$-expansion
\[ E(z) := \sum_{a \in A_+} a u(az), \]
and we remind the reader that Gekeler first obtained this form as the logarithmic derivative of $\Delta$, \cite[(8.2)]{EGinv}.

Observe that unlike the classical situation, this ``Eisenstein series'' vanishes to the order one in $u$ at the infinite cusp. This function is a Drinfeld quasi-modular form in the sense of Bosser-Pellarin \cite{VBFPqmf} in that it satisfies
\[ E|_2^1[\gamma](z) = E(z) + \frac{-1}{\pitilde}\frac{\frac{d}{dz}j(\gamma,z)}{j(\gamma,z)}, \quad \forall \gamma \in \Gamma(1);\]
see \cite[(8.4)]{EGinv} for Gekeler's original demonstration of this fact.
From this it follows for all $\gamma \in \Gamma(1)$, that we have 
\[E_\pfrak(z) := E(z) - \pfrak E(\pfrak z) \in M_2^1(\Gamma_0(\pfrak)).\] 
Further, one has
\[  z^{-2} E_\pfrak({-1}/ z) = -\frac{1}{\pfrak}E_\pfrak(z/\pfrak),\]
allowing us to see that $E_\pfrak$ is exactly single cuspidal at the zero cusp. This function is additionally a Hecke eigenform with $T_\qfrak E_\pfrak = \qfrak E_\pfrak$, for all $\qfrak$ coprime to $\pfrak$. Thus, $E_\pfrak$ gives the canonical representative for the space of cusp forms modulo double cusp forms claimed above.

\subsection{B\"ockle's Eichler-Shimura Relation and Goss' Questions} \

In \cite{GBes}, G. B\"ockle showed that the naive local $L$-factors 
\[ L_{f,\pfrak}(u)^{-1} := {1-u\lambda_\pfrak} \in \CC_\infty[u] \]
one can associate to a cuspidal Hecke eigenform $f$ with $T_\pfrak f = \lambda_\pfrak f$, for all but finitely many $\pfrak$, actually arise from something geometric --- a rank 1 $\tau$-sheaf. Further, the connection with $\tau$-sheaves and the cohomological formalism of crystals allowed B\"ockle to prove that for each $y \in \mathbb{Z}_p$, the following naive $L$-series attached to $f$
\[ L_f(x,y) := \sideset{}{'}\prod_{ \pfrak \in Spec(A)} L_{f,\pfrak}(x^{\deg \pfrak}(\frac{\pfrak}{\theta^{\deg \pfrak}})^y) \in \CC_\infty[[x]] \]
is entire in $x$; here the primed product indicates that we are ignoring a finite set of irreducibles\footnote{For more on this point see Goss' survey \cite{DGjrms} and the original article {\it ibid}.} $\pfrak$ containing those dividing the level of $f$. 

Following B\"ockle's work, the question arises of which rank 1 $\tau$-sheaves one obtains from Drinfeld modular forms via this Eichler-Shimura relation, and one has the simpler question of when the local $L$-factors of a rank 1 Drinfeld module agree (up to translation of the argument: $L_{f,\pfrak}(u) \mapsto L_{f,\pfrak}(\pfrak^k u)$, for some integer $k$ not depending on $\pfrak$) with the local $L$-factors of a cuspidal Drinfeld Hecke eigenform. The results of this note concern this latter question. 

\subsubsection{New Results}
The first purpose of this paper is to show that one may twist the eigensystems of Drinfeld Hecke eigenforms by Dirichlet characters; see Theorem \ref{heckecomm} below. 
The projection operators of \S \ref{projops} arising in this task are exactly similar to those in the classical setting over the integers, but unlike the classical case, where the cuspform's $q$-expansion coefficients are simply twisted by a character, the effect of our projection operators on Goss' $u$-expansion coefficients does not appear to be so transparent. 

This leads us to provide a sufficient condition for a Drinfeld modular form to have non-vanishing projection in Corollary \ref{Aexpprojcor}. 
For example, letting $f$ be one of $E_\pfrak, f_s$, or $\Delta$, as in \S \ref{MFexamples}, and $\chi : A \rightarrow \CC_\infty$ a Dirichlet character whose conductor is coprime to the level of $f$ with associated normalized projection $\varpi_\chi$ (the non-normalized projection is defined in Definition \ref{projdef}), we deduce in Corollary \ref{explicitformtwistcor}
 that the twisted form $\varpi_\chi f$ is non-zero and satisfies 
\[ T_\qfrak \varpi_\chi f = \left\{ \begin{array}{lll} \qfrak \chi(\qfrak) \varpi_\chi f, & & f = E_\pfrak \text{ or } f_s, \\ \qfrak^{q-1} \chi(\qfrak) \varpi_\chi f, & & f = \Delta,  \end{array} \right. \]
for all $\qfrak$ outside of a finite set of irreducibles depending on $f$ and $\chi$.

In Proposition \ref{projprop}, we see that our projection operators send cusp forms to cuspforms, but we are unable to show that they do not decrease the order of vanishing at the cusps. 
Nevertheless, we show that for the cuspform $\Delta$ just above, choosing $\chi$ appropriately, as in \cite[(76)]{DGjrms}, $\varpi_\chi \Delta$ gives an explicit, natural cuspform demonstrating that the local $L$-factors of a generic Drinfeld module of rank 1 defined over $K$ come from those of an explicit cuspidal Drinfeld Hecke eigenform. We expect, but do not show, that $\varpi_\chi \Delta$ is even double cuspidal, and thus it should provide an explicit such form answering Goss's \cite[Question 2]{DGjrms} and its generalization by explicit construction. 
We note that if the level of $\chi$ is $\nfrak$, then the level of $\varpi_\chi \Delta$ is at most $\nfrak^2$, though it very well may be a divisor of this. 
We leave open the question of the minimal weight and level at which there is a double cuspidal Hecke eigenform with eigensystem $\qfrak^k \chi(\qfrak)$, for some $k \geq 0$ independent of $\qfrak$.

Now, we remind the reader that B\"ockle has already shown that many Goss abelian $L$-functions may be obtained from single --- and not double --- cuspidal Drinfeld modular forms (i.e. the eigensystem of this form is $\lambda_\pfrak = \chi(\pfrak) \pfrak^k$ for some finite image Dirichlet character $\chi$ defined on $A$ and some integer $k \geq 0$) --- \cite[Remark 17]{DGjrms}. 
The question of explicitly constructing such forms was also posed by Goss. 
The projection operators constructed in this paper allow us to give candidates for Goss' question, but (again) so far we have not found a clean way to determine the exact order of vanishing at the cusps for our forms after projection. 

Finally, in \S \ref{eisseriessection} we construct Eisenstein series with character which are Hecke eigenforms and have the same Dirichlet-Goss abelian $L$-function as the forms in the theorem above: Proposition \ref{eiseigenprop}. Half of the Eisenstein series with character constructed below will have $A$-expansions in the sense of Petrov \eqref{Aexp1}, and the other half will have twisted $A$-expansions \eqref{Aexp3}; the reader should see the section on specializations at roots of unity in \cite{FPRParx} for more on the ``twisted uniformizers'' appearing in the later type of $A$-expansion. We close with some congruences (\S \ref{congrsection}) between certain Eisenstein series with character constructed below and forms in Petrov's family, and, after some preparation below, twisting will provide further congruences.

\subsubsection*{Thanks.} This paper began as an attempt to generalize, by first principles, certain congruences found by Pellarin and the author in \cite{Pel-PerkVMF}.
The focus shifted after G. B\"ockle heard of the constructions in this paper and directed the author to the intriguing paper \cite{DGjrms} of Goss. 
We are grateful to these three for their mentorship. 
We also heartily thank the anonymous referee for their careful reading of this script and valuable suggestions which helped to add several valuable elements herein. 

I dedicate this work to my late doctoral advisor, David Goss. He set a firm foundation in the arithmetic of Drinfeld modules, and his passion and enthusiasm can be seen and felt in the work of so many of those who have built upon it. David, you will be deeply missed. We carry your encouraging and curious spirit with us as we continue to explore this strange and fantastic world that is function field arithmetic. 

\section{Twisting Hecke eigensystems by characters}

\subsection{Dirichlet Characters and $\chi$-eigenspaces}
\subsubsection{Characters} \label{charssect}
If we write $\mfrak = \pfrak_1^{r_1}\cdots\pfrak_n^{r_n}$ for distinct monic irreducible polynomials $\pfrak_i \in A$ and positive integers $r_i$, each Dirichlet character we consider is obtained by choosing a root $\zeta_i \in \CC_\infty$ of the polynomial $\pfrak_i$, for each $i$, and sending each $a \in A$ to 
\begin{equation} \label{chardef}
 \chi(a) = a(\zeta_1)^{e_1} \cdots a(\zeta_n)^{e_n}, 
\end{equation}
for integers $0 \leq e_i < |\pfrak_i| - 1$. 
We will write $\widehat{(A/\mfrak A)^\times}$ for the group of such characters. All such characters arise from the square-free part of $\mfrak$, namely $\pfrak_1\cdots\pfrak_n$; in other words,
\[ \widehat{(A/\mfrak A)^\times} = \widehat{(A/\pfrak_1\cdots\pfrak_n A)^\times}. \]

 We call $\chi$ {\it primitive} if $0 < e_i < |\pfrak|-1$, for $i = 1,\dots,n$, for such primitive $\chi$ we call $\pfrak_1\cdots\pfrak_n$ the {\it conductor} of $\chi$.
 Finally, we will also use the notation $\chi_\zeta : A \rightarrow \FF_q(\zeta)$ for the $\FF_q$-algebra map determined by $\theta \mapsto \zeta$. 
 With this notation, the character $\chi$ from \eqref{chardef} becomes
\[ \chi = \chi_{\zeta_1}^{e_1}\cdots\chi_{\zeta_n}^{e_n}. \]

\subsubsection{Signs}
Each such character $\chi$ has a unique {\it sign} $s_\chi \in \{ 0,1, \cdots, q-2\}$ such that 
\[ \chi(\zeta) = \zeta^{s_\chi}, \quad \forall \zeta \in \FF_q^\times \subset (A/\mfrak A)^\times. \] 
From \eqref{chardef} we see that $s_\chi$ is the remainder after division of $e_1+\cdots + e_r$ by $q-1$. Notice, $s_\chi$ does not depend on the choice of roots $\zeta_i$ used to describe $\chi$. 

\subsubsection{Forms with character}
As in the classical case, $\Gamma_1(\mfrak)$ is the kernel of the homomorphism $\Gamma_0(\mfrak) \rightarrow {(A/\mfrak A)^\times}$ sending $\gamma = \left( \begin{smallmatrix} a & b \\ c & d  \end{smallmatrix} \right)$ to $d\mod\mfrak$, and via this map one has 
\[ \Gamma_0(\mfrak) / \Gamma_1(\mfrak) \cong (A/\mfrak A)^\times. \] 
For each $\gamma = \left( \begin{smallmatrix} a & b \\ c & d  \end{smallmatrix} \right) \in \Gamma_0(\mfrak)$ and each $\chi \in \widehat{(A/\mfrak A)^\times}$, we define \[\chi(\gamma) = \chi(d). \] 

For each $\chi \in \widehat{(A/\mfrak A)^\times}$ and each positive integer $l$, we define the $l$-th generalized $\chi$-eigenspace of $M_k^m(\Gamma_1(\mfrak))$ by
\[ M_k^m(\mfrak, \chi,l):= \left\{ f \in M_k^m(\Gamma_1(\mfrak)) :  ((\cdot|_k^m[\gamma]) - \chi(\gamma))^l f = 0, \text{for all } \gamma \in \Gamma_0(\mfrak) \right\}.\]
We write 
\[ M_k^m(\mfrak,\chi) := M_k^m(\mfrak,\chi,1), \]
following classical notation. We define the eigenspaces for cuspforms $S_k^m(\mfrak, \chi)$ entirely similarly. 

The need for these generalized eigenspaces can be seen via the next example and essentially arises due to the lack of non-trivial $p$-th roots of unity in $\CC_\infty$. In particular, unlike the classical setting, we will in general have a strict containment
\begin{equation*}
\oplus_{\chi} M_k^m(\mfrak, \chi) \subset M_k^m(\Gamma_1(\mfrak)),
\end{equation*}
where the direct sum runs over all $\chi \in \widehat{(A/\pfrak_1\cdots\pfrak_n A)^\times}$ for $\mfrak = \pfrak_1^{r_1}\cdots\pfrak_n^{r_n}$, as before. NB. the containment just above is an equality if and only if $\mfrak$ is square-free. 

\subsubsection{Example: Derived Eisenstein Series} \label{dereisserex}
In \S 4.4.3 of \cite{Pel-PerkVMF}, many forms in the generalized eigenspaces described above were constructed via vectorial modular forms over Tate algebras. For example, letting $\zeta \in \FF_q^{ac} \subset \CC_\infty$ with minimal polynomial $\pfrak \in A$, write
\[ a(\theta) = \sum_{n \geq 0} a^{(n)}(\zeta) \cdot (\theta - \zeta)^n, \]
by making the substitution $\theta \mapsto (\theta - \zeta) + \zeta$ and expanding the result using the binomial theorem. 
It follows from Proposition 4.22 of {\it ibid.} that the derived Eisenstein series 
\[ E_{\zeta,0} := \sideset{}{'}\sum_{a,b \in A} \frac{a^{(0)}(\zeta)}{az+b}, E_{\zeta,1} := \sideset{}{'}\sum_{a,b} \frac{a^{(1)}(\zeta)}{az+b}, \dots, E_{\zeta,n} := \sideset{}{'}\sum_{a,b} \frac{a^{(n)}(\zeta)}{az+b}\] 
span a subspace of $M_1^1(\Gamma_1(\pfrak^{n+1}))$ on which $\gamma = \left( \begin{smallmatrix} a & b \\ c & d \end{smallmatrix}  \right) \in \Gamma_{0}(\pfrak^{n+1})$ acts via
\[ \left( \begin{matrix} E_{\zeta,n}|_1^1[\gamma] \\ E_{\zeta, n-1}|_1^1[\gamma] \\ \vdots \\ E_{\zeta,0}|_1^1[\gamma]  \end{matrix} \right) = \left( \begin{matrix} d(\zeta)  & d^{(1)}(\zeta) & \cdots & d^{(n)}(\zeta) \\ 0 & d(\zeta) & \ddots & \vdots \\ \vdots & \ddots & \ddots & d^{(1)}(\zeta) \\ 0 & \cdots & 0 & d(\zeta) \end{matrix} \right) \left( \begin{matrix} E_{\zeta,n} \\ E_{\zeta, n-1} \\ \vdots \\ E_{\zeta,0}  \end{matrix} \right).\] 

Further, in Corollary 5.24 of the above reference, it is shown that away from the level each $E_{\zeta,j}$ is a Hecke eigenform with eigenvalue $\qfrak$ for $T_\qfrak$. For a similar construction in the setting of the Carlitz module, see \cite{Ma-Per}.

\subsection{Projection operators} \label{projops}
For the remainder of the article, $\afrak$, $\mfrak$, $\nfrak$, $\pfrak$ and $\qfrak$ will be monic elements in $A$, with $\pfrak, \qfrak$ additionally irreducible, and $\psi$ will denote a character of conductor $\mfrak$. We do not exclude the possibility $\mfrak = 1$ or that $\psi$ is the trivial character, i.e. $a \mapsto 1$, for all $a \in A$. In the case where both are trivial, we identify $M_k^m(\mfrak,\psi)$ and $M_k^m(\Gamma(1))$. 

\begin{definition} \label{projdef} 
Let $f: \Omega \rightarrow \CC_\infty$ be a rigid analytic function, and let $\chi \in \widehat{(A / \nfrak A)^{\times}}$. Define 
\[ \hat{\varpi}_\chi f := \sum_{|\beta| < |\nfrak|} \chi^{-1}(\beta) f|_{k}^m[\left( \begin{matrix} \nfrak & \beta \\  0 & \nfrak  \end{matrix} \right)], \]
which is again easily seen to be a rigid analytic function on $\Omega$. 
\end{definition}

\begin{lemma} \label{trianglecusplem}
Let $f:\Omega \rightarrow \CC_\infty$ be a  rigid analytic function, and suppose, for some $0 \neq \mfrak \in A$ and some $c_n,c_{n+1},c_{n+2}, \ldots \in \CC_\infty$, with $n \geq 1$, 
\[ f(z) = \sum_{i \geq n} c_i u_\mfrak(z)^i, \text{ for all $|z|_\Im$ sufficiently large.}\]
 Then, for all  $\gamma  = \left(\begin{smallmatrix} a & b \\ 0 & d \end{smallmatrix}\right) \in M_2(A)\cap\GL_2(K)$ and for all $k,m \geq 0$, there exist $c_n^\gamma, c_{n+1}^\gamma, c_{n+2}^\gamma, \ldots \in \CC_\infty$, depending on $\gamma$, such that 
 \[ f|_k^m[\gamma](z) = \sum_{i \geq n} c_{i}^\gamma u_{d\mfrak}(z)^i, \text{ for all $|z|_\Im$ sufficiently large.} \] 
\end{lemma}
\begin{proof}
Let $\gamma = \left(\begin{smallmatrix} a & b \\ 0 & d \end{smallmatrix}\right) \in M_2(A)\cap\GL_2(K)$, and assume $f(z) = \sum_{i \geq n} c_i u_\mfrak(z)^i$, as in the statement. We have
\[ u_\mfrak^i|_k^m[\gamma](z) = a^{2m}d^{2m-k} u(\frac{az+b}{\mfrak d})^i = a^{2m}d^{2m-k} u(\frac{az}{\mfrak d})^i\frac{1}{(1+u(\frac{az}{\mfrak d})\exp_C( \frac{\pitilde b}{\mfrak d}))^{i}}. \]
Expanding using geometric series (or Newton's binomial theorem), we notice that the right side above lies in $u_{d\mfrak}^{|a|i}\CC_\infty[[u_{d\mfrak}]]$. Rearranging the sums --- while noting that the $u(\frac{az}{\mfrak d})^i$ on the right side above implies that only finitely many coefficients $c_i$ contribute to each $c_j^\gamma$ and that the least power of $u_{d\mfrak}$ to appear is $u_{d\mfrak}^{|a|n}$ --- gives the result. 
\end{proof}

We have the following analog to the classical result on twisting the coefficients of modular forms by characters. 
Its proof is modeled on \cite[Proposition III.3.17]{Kob} --- see also \cite{AtLi}. While only slightly different here due to the presence of the determinant character, we present a full proof for completeness.

\begin{proposition} \label{projprop}
Let $f \in M_k^m(\mfrak,\psi)$, and let $\chi \in \widehat{(A / \nfrak A)^{\times}}$. For any monic $\afrak \in A$, divisible by both $\mfrak$ and $\nfrak^2$, we have
\[ \hat{\varpi}_\chi f  \in M_k^{m + s_{\chi}}(\afrak, \psi\chi^2). \]
Furthermore, $\hat\varpi_\chi$ sends cusp forms to cusp forms.
\end{proposition}
\begin{proof}
Let $\afrak$ be as above and view $f \in M_k^m(\Gamma_1(\afrak))$. 
Let $\gamma = \left( \begin{smallmatrix} a & b \\ c & d \end{smallmatrix} \right) \in \Gamma_0(\afrak)$. 
After a small matrix calculation, we have that $( \hat{\varpi}_\chi f)|_k^m[\gamma]$ equals
\[\sum_{|\beta| < |\nfrak|} \chi^{-1}(\beta) f|_{k}^m[\left( \begin{matrix} a+\frac{\beta c}{\nfrak} & b + \frac{x\beta c}{\nfrak^2} - \frac{\beta d - xa}{\nfrak} \\ c & d - \frac{cx}{\nfrak} \end{matrix} \right)\left( \begin{matrix} \nfrak & x \\  0 & \nfrak  \end{matrix} \right)], \]
where $|x| < |\nfrak|$ is the element such that $x \equiv \beta d a^{-1} \pmod{\nfrak}$, which exists and depends uniquely on $\beta$, as both $d$ and $a$ are coprime to $\nfrak$. 
In particular, if $\beta = 0$, then $x = 0$, and if $\beta \neq 0$, then $\beta \equiv xa d^{-1} \equiv x d^{-2} \det\gamma \pmod{\nfrak}$.
 Further, by our assumptions on $\afrak$ (so that $c$ is divisible by $\nfrak^2$), by our definition of $x$, and by multiplicativity of the determinant, we have
\[ \left( \begin{matrix} a+\frac{\beta c}{\nfrak} & b + \frac{x\beta c}{\nfrak^2} - \frac{\beta d - xa}{\nfrak} \\ c & d - \frac{cx}{\nfrak} \end{matrix} \right) \in \Gamma_0(\afrak) \subset \Gamma_0(\mfrak). \]
Hence, using $f \in M_k^m(\mfrak,\psi)$ and $\beta \equiv  x d^{-2} \det\gamma \pmod{\nfrak}$,
\[ ( \hat{\varpi}_\chi f)|_k^m[\gamma] := \frac{\psi\chi^2(\gamma)}{(\det\gamma)^{s_\chi}} \sum_{|\beta| < |\nfrak|} \chi^{-1}(x) f|_{k}^m[\left( \begin{matrix} \nfrak & x \\  0 & \nfrak  \end{matrix} \right)] = \frac{\psi\chi^2(\gamma)}{(\det\gamma)^{s_\chi}} \hat{\varpi}_\chi f, \]
giving the required functional equation.

Finally, we analyze what happens at the cusps. It is enough  to understand $f|_k^m[\left( \begin{smallmatrix} \nfrak & \beta \\  0 & \nfrak  \end{smallmatrix} \right)\gamma]$, for each $|\beta| < |\nfrak|$ and each $\gamma \in \Gamma(1)$. For $\gamma = \left( \begin{smallmatrix} a & b \\ c & d \end{smallmatrix} \right) \in \Gamma(1)$ and each such $\beta$, we have
\[ \left( \begin{matrix} \nfrak & \beta \\  0 & \nfrak  \end{matrix} \right)\left( \begin{matrix} a & b \\ c & d \end{matrix} \right) = \left( \begin{matrix} \nfrak a + \beta c & \nfrak b + \beta d \\  \nfrak c & \nfrak d  \end{matrix} \right).\]
If $c = 0$, since $\hat\varpi_\chi f$ is $A$-periodic, the $u$-expansion at infinity is preserved and there is nothing to show. 
If $\nfrak a + \beta c = 0$, we apply Lemma \ref{trianglecusplem} to $f|_k^m[\left( \begin{smallmatrix} 0 & 1 \\  1 & 0  \end{smallmatrix} \right)]$ and $\left( \begin{smallmatrix}  \nfrak c & \nfrak d \\ 0 & \nfrak b + \beta d  \end{smallmatrix} \right)$, noting that $\nfrak c(\nfrak b + \beta d) \in \nfrak^2 \FF_q^\times$, which yields a $u$-expansion that vanishes at infinity. 
If both $c$ and $\nfrak a + \beta c$ are non-zero, we let $g \in A_+$ be their greatest common divisor, and we choose $x,y \in A$ such that 
\[\gamma_0 := \left( \begin{matrix} y & x \\  -\nfrak c/ g & (\nfrak a + \beta c)/g  \end{matrix} \right) \in \text{SL}_2(A).\] 
Then,
\[ \left( \begin{matrix} \nfrak & \beta \\  0 & \nfrak  \end{matrix} \right)\left( \begin{matrix} a & b \\ c & d \end{matrix} \right) = \left( \begin{matrix} y & x \\  -\nfrak c/ g & (\nfrak a + \beta c)/g  \end{matrix} \right)^{-1} \left( \begin{matrix} g & y(\nfrak b + \beta d) \\  0 & \nfrak^2/g  \end{matrix} \right). \]
Thus applying Lemma \ref{trianglecusplem} to $f|_k^m[\gamma_0^{-1}]$ and $\left( \begin{smallmatrix} g & y(\nfrak b + \beta d) \\  0 & \nfrak^2/g  \end{smallmatrix} \right)$ again yields a $u$-expansion which vanishes at infinity and finishes the proof.
\end{proof}

\subsubsection{Non-vanishing}
We begin will a lemma on the convolution of two Dirichlet characters, as defined above, or what amounts to some kind of function field Jacobi sums. 
Unfortunately, we do not see a relation with Thakur's Gauss or Jacobi sums, contrary to what one might expect from comparison with the classical case. 
We focus on such convolutions for square-free moduli, as the Dirichlet characters we consider in this paper have such conductors. 

Consider $\mathfrak{n} = \mathfrak{p}_1\cdots\mathfrak{p}_r$ and $\zeta_i$ such that $\mathfrak{p}_i(\zeta_i) = 0$. Let 
\begin{eqnarray} \label{chi1chi2eq} \chi_1 := \prod_{i = 1}^r\chi_{\zeta_i}^{j_i} & \text{ and } & \chi_2 := \prod_{i = 1}^r\chi_{\zeta_i}^{k_i}, \\ \nonumber & \text{ with } & 1 \leq k_i, j_i < |\mathfrak{p}_i| - 1, \text{ for all } i = 1,2,\dots,r,
\end{eqnarray} 
so that these are characters of the same conductor $\mathfrak{n}$. 

For such characters and each $|\delta|< |n|$, define
\[ (\chi_1*\chi_2)(\delta) := \sum_{|a| < |\mathfrak{n}|} \chi_1(a) \chi_2(\delta - a). \]
In the next result, when writing $\chi_1\chi_2$ we shall mean
\[{\chi_1\chi_2} := \prod_{i = 1}^r \chi_{\zeta_i}^{\overline{j_i + k_i}}, \]
where $\overline{j_i+k_i}$ is the unique representative in $0,1,\dots,|\pfrak_i|-2$ for $j_i+k_i \pmod{|\pfrak_i|-1}$. This $\chi_1\chi_2$ is the primitive character associated to the point-wise product function \newline $a \mapsto \chi_1(a)\chi_2(a)$. 

\begin{lemma} \label{Jsumlem}
For characters $\chi_1,\chi_2$, as in \eqref{chi1chi2eq}, and for each $|\delta| < |n|$, we have
\[ (\chi_1*\chi_2)(\delta) = ({\chi_1\chi_2})(\delta)\cdot\prod_{i = 1}^r (-1)^{1 - j_i} {k_i \choose |\mathfrak{p}_i| - 1 - j_i}. \]
\end{lemma}
\begin{proof}
Letting $\chi_1$ and $\chi_2$ as above, we have
\begin{eqnarray*}
(\chi_1*\chi_2)(\delta) &:=& \sum_{|a| < |\mathfrak{n}|} \prod_{i = 1}^r a(\zeta_i)^{j_i} (\delta(\zeta_i)-a(\zeta_i))^{k_i} \\
&=& \sum_{l_1 = 0}^{k_1}\cdots \sum_{l_r = 0}^{k_r} \left( \prod_{i = 1}^r (-1)^{l_i} {k_i \choose l_i} \delta(\zeta_i)^{k_i - l_i} \right) \sum_{|a| < |\mathfrak{n}|} \prod_{i = 1}^r a(\zeta_i)^{j_i+l_i}.
\end{eqnarray*}
Now, the sum over $|a|<|\nfrak|$ is non-zero if and only if $j_i + l_i = |\pfrak_i| - 1$ for each $i = 1,2,\dots,r$. However, this condition will not be satisfied if $k_i+j_i < |\pfrak_i| - 1$ for some $i$. Nevertheless, in this situation, the binomial coefficient on the right side of the desired identity vanishes by the usual convention, and the identity holds trivially. 
If we do have $k_i+j_i \geq |\pfrak_i|-1$ for each $i = 1,2,\dots,r$, so that there exists a $0 \leq l_i \leq k_i$ such that $j_i + l_i = |\pfrak_i| - 1$ for each $i = 1,2,\dots,r$, then the sum over $|a|<|\nfrak|$ equals 
\[ |(A / \nfrak A)^\times| = (|\pfrak_1| - 1)\cdots (|\pfrak_r| - 1) = (-1)^r, \]
and $k_i - l_i = k_i + j_i - (|\pfrak_i|-1)$ is non-negative and strictly less than $|\pfrak_i|-1.$ Hence, $k_i - l_i = \overline{k_i + j_i}$, and inputting all of this data, we obtain the identity.
\end{proof}

From the previous result we see that the (non-)vanishing of these positive characteristic Jacobi sums has a very combinatorial description in terms of the base $p$ expansions of the $k_i$ and $j_i$, by Lucas' Theorem. 
 
 \begin{proposition}
 Let $\chi_1,\chi_2 \in \widehat{(A/\mathfrak{n}A)^\times}$, as in \eqref{chi1chi2eq}, and let $f \in M_k^m(\mathfrak{m},\psi)$. Let $s_1 := s_{\chi_1}$. 
 
 We have
 \[ \hat\varpi_{\chi_2} \hat\varpi_{\chi_1} f  = \mathfrak{n}^{2 s_{1}+2m - k} \cdot \left( \prod_{i = 1}^r (-1)^{j_i+1} {|\pfrak_i| - 1 - k_i \choose j_i} \right)\hat{\varpi}_{{\chi_2\chi_1}}f. \]
 
 In particular, when $\chi_1 = \chi_2^{-1}$, we have 
 \[\hat\varpi_{\chi_1^{-1}} \hat\varpi_{\chi_1} f = (-1)^{j_1+\cdots+j_r+r} \mathfrak{n}^{2 s_{1}+2m - k} \sum_{ |\delta|<|\mathfrak{n}| } f|_k^m [\left( \begin{matrix} \mathfrak{n} & \delta \\ 0 & \mathfrak{n}  \end{matrix} \right)]. \]
 \end{proposition}
\begin{proof}
Consider $\hat\varpi_{\chi_2} \hat\varpi_{\chi_1} f$, for characters $\chi_1, \chi_2 \in \widehat{(A/\mathfrak{n}A)^\times}$. Taking into account that $\varpi_{\chi_1} f \in M_k^{m+s_{1}}(\mathfrak{l},\psi\chi_1^2)$ when $f \in M_k^m(\mathfrak{m},\psi)$, where $\mathfrak{l} = \text{lcm}(\mathfrak{m},\nfrak^2)$, we have
\begin{eqnarray*}
\hat\varpi_{\chi_2} \hat\varpi_{\chi_1} f &=& \mathfrak{n}^{2 s_{1}}\sum_{|\beta| < |\mathfrak{n}|}\sum_{|\alpha| < |\mathfrak{n}|} \chi_2^{-1}(\beta)\chi_1^{-1}(\alpha) f|_k^m [\left( \begin{matrix} \mathfrak{n} & \alpha \\ 0 & \mathfrak{n}  \end{matrix} \right) \left( \begin{matrix} \mathfrak{n} & \beta \\ 0 & \mathfrak{n}  \end{matrix} \right)] \\
&=& \mathfrak{n}^{2 s_{1}+2m - k}\sum_{\alpha,\beta} \chi_2^{-1}(\beta)\chi_1^{-1}(\alpha) f|_k^m [\left( \begin{matrix} \mathfrak{n} & \alpha + \beta \\ 0 & \mathfrak{n}  \end{matrix} \right)] \\
&=& \mathfrak{n}^{2 s_{1}+2m - k} \sum_{ |\delta|<|\mathfrak{n}|} \left( \sum_{|\alpha| < |\mathfrak{n}|} \chi_2^{-1}(\delta - \alpha)\chi_1^{-1}(\alpha) \right) f|_k^m [\left( \begin{matrix} \mathfrak{n} & \delta \\ 0 & \mathfrak{n}  \end{matrix} \right)] \\
&=& \mathfrak{n}^{2 s_{1}+2m - k} \cdot \left( \prod_{i = 1}^r (-1)^{j_i+1} {|\pfrak_i| - 1 - k_i \choose j_i} \right)\hat{\varpi}_{{\chi_2\chi_1}}f .
\end{eqnarray*}
\end{proof}

\begin{corollary} \label{Aexpprojcor}
Let $\nfrak \in A_+$ be square-free. Suppose, for some $1 \leq i \leq q$, the form $f \in M_k^m(\mathfrak{m},\psi)$ has an $A$-expansion of the shape
\[ f(z) = \sum_{a \in A_+} c_a u(az)^i, \text{ for some } c_a \in \CC_\infty, \]
and suppose further that there exists an $a \in A_+$, coprime to $\nfrak$, such that $c_a \neq 0$. Then, for each primitive $\chi \in  \widehat{(A/\mathfrak{n}A)^\times}$, we have
\[ \hat\varpi_{\chi} f \neq 0.\]
\end{corollary}
\begin{proof}
Following Gekeler's computation leading to \cite[(7.3)]{EGinv}, we obtain
\[ \sum_{ |\delta|<|\mathfrak{n}| } u^i|_k^m [\left( \begin{matrix} \mathfrak{n} & \delta \\ 0 & \mathfrak{n}  \end{matrix} \right)](z) = \left\{ \begin{array}{ll} G_{\nfrak, i}(\nfrak u(\nfrak a z)) & \text{if } (\nfrak,a) = 1, \\ 0 & \text{otherwise}, \end{array} \right. \] 
where $G_{\nfrak, i}$ is the $i$-th Goss polynomial for the lattice of Carlitz $\nfrak$-torsion. Thus, after observing that $G_{i,\nfrak}(X) = X^i$, for $1 \leq i \leq q$, we obtain the identity \[ \sum_{ |\delta|<|\mathfrak{n}| } f|_k^m [\left( \begin{matrix} \mathfrak{n} & \delta \\ 0 & \mathfrak{n}  \end{matrix} \right)](z) = \nfrak^i\sum_{\substack{a \in A_+ \\ (a,\nfrak) = 1}} c_a u(a \nfrak z)^i. \]

In \cite[Th. 3.1]{BLadm11} it is shown that a rigid analytic function with non-vanishing $A$-expansion at infinity is non-zero, and, thus, by our assumption on the existence of an $a$ coprime to $\nfrak$ such that $c_a \neq 0$, we deduce the non-vanishing of the $A$-expansion on the right side of the identity just above. 
 Finally then, the non-vanishing of $\hat\varpi_{\chi} f$ follows from the last proposition.
\end{proof}

\subsubsection{Commutation Relations}
We note here that, for $f \in M_k^m(\mfrak, \psi)$, the Hecke action of Definition \ref{Heckeactiondef} becomes
\begin{equation} \label{heckewchar}
T_\qfrak f = \qfrak^{k-m}(\psi(\qfrak)f|_k^m[\left( \begin{matrix} \qfrak & 0 \\  0 & 1  \end{matrix} \right)] + \sum_{|\beta|<|\qfrak|} f|_k^m[\left( \begin{matrix} 1 & \beta \\ 0 & \qfrak \end{matrix} \right)]).
\end{equation}

\begin{proposition}
Let $f \in M_k^m(\mfrak,\psi)$, and let $\chi \in \widehat{(A/\nfrak A )^\times}$. For all irreducible $\qfrak$ with $(\qfrak,\nfrak \mfrak) = 1$, we have
\[ T_\qfrak  \hat{\varpi}_\chi f = \chi(\qfrak)   \hat{\varpi}_\chi T_\qfrak f,\]
where the Hecke operators act on the proper weight, level, and type for $\hat{\varpi}_\chi f$ and $f$, respectively. 
\end{proposition}
\begin{proof}
We have
\begin{eqnarray*} T_\qfrak  \hat{\varpi}_\chi f &=& \qfrak^{k-(m+s_\chi)} \left( \psi\chi^2(\qfrak) \sum_{|\alpha| < |\nfrak|}\chi^{-1}(\alpha) f |_k^{m} [\left( \begin{smallmatrix} \nfrak & \alpha \\  0 & \nfrak  \end{smallmatrix} \right)] |_k^{m+s_\chi} [\left( \begin{smallmatrix} \qfrak & 0 \\ 0 & 1  \end{smallmatrix} \right)] \right. \\
&&+ \left. \sum_{|\beta| < |\qfrak|}\sum_{|\alpha| < |\nfrak|}\chi^{-1}(\alpha) f |_k^{m} [\left( \begin{smallmatrix} \nfrak & \alpha \\  0 & \nfrak  \end{smallmatrix} \right)] |_k^{m+s_\chi} [\left( \begin{smallmatrix} 1 & \beta \\ 0 & \qfrak  \end{smallmatrix} \right)] \right) \\
&=& \qfrak^{k-m} \left( \psi\chi^2(\qfrak) \sum_{|\alpha| < |\nfrak|}\chi^{-1}(\alpha) f |_k^{m} [\left( \begin{smallmatrix} \nfrak & \alpha \\  0 & \nfrak  \end{smallmatrix} \right)\left( \begin{smallmatrix} \qfrak & 0 \\ 0 & 1  \end{smallmatrix} \right)] \right. \\
&&+ \left. \sum_{|\beta| < |\qfrak|}\sum_{|\alpha| < |\nfrak|}\chi^{-1}(\alpha) f |_k^{m} [\left( \begin{smallmatrix} \nfrak & \alpha \\  0 & \nfrak  \end{smallmatrix} \right)\left( \begin{smallmatrix} 1 & \beta \\ 0 & \qfrak  \end{smallmatrix} \right)] \right),
\end{eqnarray*}
where we have applied \eqref{heckewchar} using Proposition \ref{projprop}. 

We check the commutativity of the matrices involved. We have
\begin{eqnarray*}
\left( \begin{matrix} \nfrak & \alpha \\  0 & \nfrak  \end{matrix} \right)\left( \begin{matrix} \qfrak & 0 \\ 0 & 1  \end{matrix} \right) = \left( \begin{matrix} 1 & \alpha(1 - \qfrak\qfrak^*)/\nfrak \\  0 & 1  \end{matrix} \right) \left( \begin{matrix} \qfrak & 0 \\ 0 & 1  \end{matrix} \right) \left( \begin{matrix} 1 & k_\alpha^* \\  0 & 1  \end{matrix} \right) \left( \begin{matrix} \nfrak & r_\alpha^* \\  0 & \nfrak  \end{matrix} \right),
\end{eqnarray*}
where $\qfrak^*$ is such that $\qfrak \qfrak^* \equiv  1 \pmod{\nfrak}$ and $|r_\alpha^*|<|\nfrak|$ and $k_\alpha^*$ are such that $\alpha \qfrak^* = r_\alpha^* + k_\alpha^* \nfrak$. Observe that as $\alpha$ runs over the set $\{|\alpha| < |\nfrak|\}$, $r_\alpha$ does too and $r_0 = 0$.

Similarly, 
\begin{eqnarray*}
\left( \begin{matrix} \nfrak & \alpha \\  0 & \nfrak  \end{matrix} \right)\left( \begin{matrix} 1 & \beta \\ 0 & \qfrak  \end{matrix} \right) =  \left( \begin{matrix} 1 & \beta \\ 0 & \qfrak  \end{matrix} \right) \left( \begin{matrix} 1 & k_\alpha \\  0 & 1  \end{matrix} \right) \left( \begin{matrix} \nfrak & r_\alpha \\  0 & \nfrak  \end{matrix} \right),
\end{eqnarray*}
where, as before, $|r_\alpha|<|\nfrak|$ is such that $\alpha \qfrak = r_\alpha + k_\alpha \nfrak$.

Thus, $ \frac{1}{\qfrak^{k-m}} T_\qfrak  \hat{\varpi}_\chi f$ becomes 
\[\psi\chi^2(\qfrak) \sum_{|\alpha| < |\nfrak|}\chi^{-1}(\alpha) f |_k^m [\left( \begin{smallmatrix} \qfrak & 0 \\ 0 & 1  \end{smallmatrix} \right)\left( \begin{smallmatrix} \nfrak & r_\alpha^* \\  0 & \nfrak  \end{smallmatrix} \right)] + \sum_{|\alpha| < |\nfrak|}\chi^{-1}(\alpha) \sum_{|\beta| < |\qfrak|} f |_k^m [\left( \begin{smallmatrix} 1 & \beta \\ 0 & \qfrak  \end{smallmatrix} \right)\left( \begin{smallmatrix} \nfrak & r_\alpha \\  0 & \nfrak  \end{smallmatrix} \right)],\]
where in the first sum we have used the modularity of $f$ and $f|[\left( \begin{smallmatrix} \qfrak & 0 \\ 0 & 1  \end{smallmatrix} \right)]$ and in the second sum we have used the modularity of $\sum_{|\beta| < |\qfrak|} f |_k^m [\left( \begin{smallmatrix} 1 & \beta \\ 0 & \qfrak  \end{smallmatrix} \right)]$. Finally, we observe that $\chi^{-1}(\alpha) = \chi^{-1}(\qfrak)\chi^{-1}(r_\alpha^*) = \chi(\qfrak)\chi^{-1}(r_\alpha)$ and obtain
\[T_\qfrak  \hat{\varpi}_\chi f = \qfrak^{k-m} \chi(\qfrak) \hat{\varpi}_\chi\left(\psi(\qfrak)f|_k^m[\left( \begin{smallmatrix} \qfrak & 0 \\ 0 & 1  \end{smallmatrix} \right)] +  \sum_{|\beta| < |\qfrak|} f |_k^m [\left( \begin{smallmatrix} 1 & \beta \\ 0 & \qfrak  \end{smallmatrix} \right)]\right),\]
which is what we wanted to show.
\end{proof}

The following result is immediate and resolves the second part of \cite[Ques. 1]{DGjrms} on twisting a form's Hecke eigensystem by Dirichlet characters, whenever $\hat{\varpi}_\chi f$ is non-zero. 

\begin{theorem} \label{heckecomm}
Let $f$, $\chi$, and $\qfrak$ be as in the previous proposition. If $T_\qfrak f = \lambda_\qfrak f$ for some $\lambda_\qfrak \in \CC_\infty$, then 
\[T_\qfrak  \hat{\varpi}_\chi f = \chi(\qfrak) \lambda_\qfrak  \hat{\varpi}_\chi f.\] 
\hfill \qed
\end{theorem}

Choosing the character $\chi$ appropriately, as in \cite[76)]{DGjrms}, the next result affirmatively answers Goss' \cite[Question 2]{DGjrms} (up to showing double cuspidality!) and shows more generally that, for all but finitely many places, the local $L$-factors of any Drinfeld module of rank one defined over $K$ may be obtained from the local $L$-factors of some explicit cuspidal Drinfeld modular form of weight as small as two.

\begin{corollary} \label{explicitformtwistcor}
For each square-free $\nfrak \in A_+$ and for each primitive $\chi \in  \widehat{(A/\mathfrak{n}A)^\times}$, the forms $\hat\varpi_{\chi} f$, with $f = f_s$ in Petrov's family,  $f = \Delta$, or $f = E_\pfrak$, with $\pfrak$ an arbitrary monic irreducible in $A$, are all non-vanishing cuspidal Hecke eigenforms with eigensystems $\chi(\qfrak)\qfrak$, $\chi(\qfrak)\qfrak^{q-1}$, and $\chi(\qfrak)\qfrak$, respectively. 
\end{corollary}

\subsubsection{Normalized Projections for Square-free Levels} \label{GTsection} 
In order to keep the coefficient field of the $u$-expansion of $ \hat{\varpi}_\chi f$ as close as possible to that of the $u$-expansion for $f$, we must employ a slight generalization of the Gauss sums introduced by Thakur \cite{DTgtsums} for Carlitz torsion extensions.

We direct the reader to \cite[\S 2.3]{BA-FPjnt} for the basic definitions and properties of Gauss-Thakur sums associated to characters of square-free moduli. Following their notation, we will write $g(\chi)$ for the Gauss-Thakur sum associated to $\chi \in \widehat{(A/\nfrak A)^\times}$, with $\nfrak$ square-free. Let $L$ be a finite extension of  $K$, and let $R \subset L$ be the integral closure of $A \subset K$. For each Dirichlet character $\chi$, we write $L(\chi)$ for the smallest field extension of $L$ containing the values of $\chi$, and we write $R[\chi] \subset L(\chi)$ for the integral closure of $A \subset K$. 

\begin{definition}
For rigid analytic functions $f:\Omega \rightarrow \CC_\infty$ and primitive $\chi \in \widehat{(A / \nfrak A)^{\times}}$ with $\nfrak$ square-free, let
\[\varpi_\chi f := \nfrak^{k-2m-1} g(\chi^{-1}) \hat{\varpi}_\chi f.\] 
\end{definition}

We remind the reader that for all positive integers $i$ and all $k \geq 0$, the binomial coefficient ${-i \choose k}$ is an integer, and we use the same symbol for its reduction modulo $p$.

\begin{proposition} \label{normprojui}
For each $i\geq 1$, we have
\[ \varpi_\chi (u^i) = u^i \sum_{\substack{k \geq 1 \\ k \equiv s_\chi \pmod{q-1}}} {-i \choose k} \left( \frac{g(\chi^{-1})}{\nfrak}\sum_{|\beta| < |\nfrak|} \chi^{-1}(\beta)  \exp_C(\frac{\pitilde \beta}{\nfrak})^k\right)u^k.\] 
\end{proposition}
\begin{proof}
By \eqref{expueq},
\begin{equation} \label{linearfrac}
u(z + a/\nfrak) = \frac{u(z)}{\exp_C(\frac{\pitilde a}{\nfrak}) u(z) + 1}.
\end{equation}
Thus, by Newton's Binomial Theorem, we have
\begin{eqnarray*}
\varpi_\chi (u^i) &=& \frac{g(\chi^{-1})}{\nfrak}\sum_{|\beta| < |\nfrak|} \chi^{-1}(\beta) (\frac{u}{\exp_C(\frac{\pitilde \beta}{\nfrak})u + 1})^i \\
&=& u^i \sum_{k \geq 0} {-i \choose k} \left( \frac{g(\chi^{-1})}{\nfrak} \sum_{|\beta| < |\nfrak|} \chi^{-1}(\beta)  \exp_C(\frac{\pitilde \beta}{\nfrak})^k\right)u^k.
\end{eqnarray*}
Now the sums 
\[\sum_{|\beta| < |\nfrak|} \chi^{-1}(\beta) \exp_C(\frac{\pitilde \beta}{\nfrak})^k, \text{ with } k \geq 0,\]
appearing in the line above, are integral over $A$ and lie in the $\chi$-eigenspace of the extension $K(\chi,\exp_C(\pitilde/\nfrak)) / K(\chi)$ under the action of Galois which, by \cite[Lemma 16]{BA-FPjnt} is generated by the Gauss-Thakur sum $g(\chi)$. 
Thus, after multiplying by $g(\chi^{-1})/\nfrak$, the reflection formula \cite[Proposition 15.2]{BA-FPjnt} tells us that we are in the integral extension of $A$ obtained by adjoining the values of $\chi$. When $k = 0$, this sum vanishes by orthogonality of characters, and the vanishing of $ \sum_{|\beta| < |\nfrak|} \chi^{-1}(\beta)  \exp_C(\frac{\pitilde \beta}{\nfrak})^k$ for $k \not\equiv s_\chi \pmod{q-1}$ is seen by making the change of variables $\beta \mapsto \xi\beta$, for a primitive multiplicative generator $\xi \in \FF_q^\times$.
\end{proof}

\begin{remark}
We noticed in the proof just above that the sums
\[ s(\chi,k) := \sum_{|\beta| < |\nfrak|} \chi^{-1}(\beta) \exp_C(\frac{\pitilde \beta}{\nfrak})^k, \text{ with } k \geq 0,\]
vanish trivially for $k = 0$ and for $k \not\equiv s_\chi \pmod{q-1}$. This non-congruence is however not implied by the vanishing of these sums. Indeed, a quick computation with SageMath \cite{SAGE} reveals that classifying the positive integers $k \equiv s_\chi \pmod{q-1}$ for which $s(\chi,k)$ is non-zero might not be easy. 

For example, for $q = 5$ and prime conductor $\nfrak = \theta^2+2$ with $\zeta$ such that $\nfrak(\zeta) = 0$, the full list of $1 \leq i,j \leq 23$ such that 
\[ \sum_{|\beta| < |\nfrak|} \beta(\zeta)^{|\nfrak|-1-i} \exp_C(\pitilde \beta / \nfrak)^j \neq 0 \]
is given by
\[ \begin{matrix} [j,i]: \\  \\ [1, 1], [1, 5], \\ [2, 2], [2, 6], [2, 10], \\ [3, 3], [3, 7], [3, 11], [3, 15], \\ [4, 4], [4, 8], [4, 12], [4, 16], [4, 20], \\ [5, 1], [5, 5], \\ [6, 2], [6, 6], [6, 10], \\ [7, 3], [7, 7], [7, 11], [7, 15], \\ [8, 4], [8, 8], [8, 12], [8, 16], [8, 20], \\ [9, 1], [9, 5], [9, 9], [9, 13], [9, 17], [9, 21], \\ [10, 2], [10, 6], [10, 10], \\ [11, 3], [11, 7], [11, 11], [11, 15], \\ [12, 4], [12, 8], [12, 12], [12, 16], [12, 20], \\ [13, 1], [13, 5], [13, 9], [13, 13], [13, 17], [13, 21], \\ [14, 2], [14, 6], [14, 10], [14, 14], [14, 18], [14, 22], \\ [15, 3], [15, 7], [15, 11], [15, 15], \\ [16, 4], [16, 8], [16, 12], [16, 16], [16, 20], \\ [17, 1], [17, 5], [17, 9], [17, 13], [17, 17], [17, 21], \\ [18, 2], [18, 6], [18, 10], [18, 14], [18, 18], [18, 22], \\ [19, 3], [19, 7], [19, 11], [19, 15], [19, 19], [19, 23], \\ [20, 4], [20, 8], [20, 12], [20, 16], [20, 20], \\ [21, 1], [21, 5], [21, 9], [21, 13], [21, 17], [21, 21], \\ [22, 2], [22, 6], [22, 10], [22, 14], [22, 18], [22, 22],  \\ [23, 3], [23, 7], [23, 11], [23, 15], [23, 19], [23, 23]. \end{matrix} \]
\hfill \qed
\end{remark}

We remind the reader that any $f \in M_k^m(\mfrak,\psi)$ is $A$-periodic and thus has a power series expansion at infinity in the parameter $u$. The next result deals solely with the $u$-expansion at infinity of such forms. 

\begin{corollary} \label{normalizationlem}
Let $f \in M_k^m(\mfrak,\psi)$ and $\chi \in \widehat{(A / \nfrak A)^{\times}}$, primitive with $\nfrak$ square-free.

If $f \in R[[u]]$, for some integral extension $R$ over $A$, then $\varpi_\chi f \in R[\chi][[u]]$. 
\end{corollary}
\begin{proof}
Using the identity of Proposition \ref{normprojui}, summing over $i$, since $u^i$ divides $\varpi_\chi (u^i)$, as demonstrated above, we see that only finitely many terms contribute to any given power of $u$. Thus, the claim on that the coefficients lie in $R[\chi]$ is clear, and convergence does not pose an issue when $u$ is sufficiently small.
\end{proof}

\begin{remark}
Again, we direct the reader to \cite[\S 5]{FPRParx} where further study is made on the rational functions of $u$ given by $\varpi_\chi(u)$ using the point of view of polynomial interpolation. We expect that by using the rational functions studied in {\it ibid.} and the formalism of hyperderivatives a clean formula can be given, akin to that of \cite[(7.3)]{EGinv}, for the effect of these twisting operators  $\varpi_\chi$ on the $u$-expansions of Goss. Since we will not use it below, we refrain from this undertaking here.
\end{remark}

\section{Eisenstein series with character and Congruences} \label{eisseriessection}
To simplify the presentation of this final section, we restrict our attention to prime levels $\pfrak$. We direct the reader to the introduction for the definition of the Eisenstein series for $\Gamma(\pfrak)$. 

As in the classical characteristic zero setting over the integers, for $(0,a) \in V_\pfrak$ (see \eqref{Vmdefeq} for the definition of this set) and $\gamma = \left( \begin{smallmatrix} * & * \\ * & d_\gamma  \end{smallmatrix} \right) \in \Gamma_0(\pfrak)$, we have 
\[(0,a)\gamma = (0,ad_\gamma).\] 
In particular, $E_{(0,a)}^{(k)} \in M_k^0(\Gamma_1(\pfrak))$, for each $a \in (A/\pfrak A)^\times$. 
Thus, for each character $\chi \in \widehat{(A/\pfrak A)^\times}$, we have
\[E^{(k)}_\chi := \sum_{a \in (A/\pfrak A)^\times} \chi^{-1}(a) E^{(k)}_{(0,a)} \in M_k(\pfrak,\chi). \] 
By consideration of the matrices $\left( \begin{smallmatrix} \zeta & 0 \\ 0 & \zeta  \end{smallmatrix} \right) \in \Gamma_0(\pfrak)$ we see that if $E_\chi^{(k)} \neq 0$, we must have 
\begin{equation} \label{eisnonvaneq} 
s_\chi \equiv -k \pmod{q-1}. 
\end{equation}
Our next aim is to show, by means of an $A$-expansion at the zero cusp, that for characters satisfying the congruence \eqref{eisnonvaneq} these Eisenstein series are non-zero. 

Let $W_\pfrak := \left( \begin{smallmatrix} 0 & -1 \\ \pfrak & 0  \end{smallmatrix} \right)$, and observe that 
\begin{equation} \label{Wpconjeq}
W_\pfrak\left( \begin{matrix} a & b \\ c & d  \end{matrix} \right)W_{\pfrak}^{-1} = \left( \begin{matrix} d & -c/\pfrak \\ -\pfrak b & a \end{matrix} \right) \in \Gamma_0(\pfrak).
\end{equation}
Thus, $W_\pfrak$ is in the normalizer of $\Gamma_0(\pfrak)$ in $\GL_2(K)$. 
For a modular form $f \in M_k^m(\pfrak, \chi)$, we call $f|_k^m[W_\pfrak]$ its {\it Fricke transform}.

\begin{lemma}
Suppose $f \in M_k^m(\pfrak, \chi)$, then $f|_k^m[W_\pfrak] \in M_k^{m - s_\chi}(\pfrak, \chi^{-1})$. 

In particular, when $E_\chi^{(k)} \neq 0$, we have $E_{\chi}^{(k)}|_{k}[W_\pfrak] \in M_k^{k}(\pfrak, \chi^{-1})$.
\end{lemma}
\begin{proof}
For $\gamma = \left( \begin{smallmatrix} a & b \\ c & d  \end{smallmatrix} \right) \in \Gamma_0(\pfrak)$, we have $ad \equiv \det \gamma \pmod{\pfrak}$ from which it follows that 
\[ \chi(a) \equiv \chi^{-1}(d) (\det \gamma)^{s_{\chi}} = \chi^{-1}(\gamma) (\det \gamma)^{s_{\chi}}.\] 
Thus, by the computation of the previous line, for 
$\gamma = \left( \begin{smallmatrix} a & b \\ c & d  \end{smallmatrix} \right) \in \Gamma_0(\pfrak)$,
\[(f|_{k}^m[ W_\pfrak ])|_k^m[\left( \begin{smallmatrix} a & b \\ c & d  \end{smallmatrix} \right)] = (f|_{k}^m[\left( \begin{smallmatrix} d & -c/\pfrak \\ -\pfrak b & a \end{smallmatrix} \right)])|_k^m[W_\pfrak] = \chi^{-1}(\gamma)(\det \gamma)^{s_\chi}f|_k^m[W_\pfrak].\]
Finally, holomorphy only needs to be checked at the cusps represented by zero and infinity, and this is easy. 
\end{proof}

\subsection{$A$-expansions}
The path toward $A$-expansions for Eisenstein series for $\Gamma(1)$ was discovered by Goss. For all $k \geq 0$, he introduced certain polynomials $G_k \in K[z]$, now referred to as Goss polynomials, for the lattice $\pitilde A$; see \cite{DGcrelles,DGbams}. We follow the notation of Gekeler, as in \cite[\S 4]{EGinv}. The fundamental property of the Goss polynomials is that for all $k \geq 1$ one has
\[\frac{1}{\pitilde^k}\sum_{a \in A} \frac{1}{(z-a)^k} = G_k(u(z));\]
see e.g. \cite[(3.4)]{EGinv} for further explanation. We shall use that
\begin{equation} \label{Gkeq}
G_k(u(\zeta z)) = \zeta^{-k}G_k(u(z)), \forall \zeta \in \FF_q^\times, k \geq 1.
\end{equation}

To follow, we let $K(\chi)$ be the finite extension of $K$ obtained by adjoining the values of $\chi$. The next result demonstrates that a subset of the Goss-Eisenstein series with character have $A$-expansions in the sense of Petrov. In fact, when $k = 1$, they could be considered as the ``reduction of the $A$-expansion for some form in Petrov's special family modulo $\pfrak$; we will make this precise in Proposition \ref{SFcongs} below.

\begin{proposition}
If $k$ be a positive integer and $\chi \in  \widehat{(A/\pfrak A)^\times}$ satisfies \eqref{eisnonvaneq}, then
\begin{equation}\label{Aexp1} -\frac{\pfrak^k}{\pitilde^k} (E_{\chi}^{(k)}|_{k}[W_\pfrak])(z) = \sum_{c \in A_+} \chi^{-1}(c) G_k(u(c z)) \in K(\chi)[[u]]. \end{equation}

In particular, for such $\chi$, the functions $E_{\chi}^{(k)}$ and $E_{\chi}^{(k)}|_k[W_\pfrak]$ are non-zero. 
\end{proposition}
\begin{proof} 
Let $\chi$ and $k$ be as in the statement. We have
\begin{eqnarray*}
 (E_{\chi}^{(k)}|_{k}[W_\pfrak])(z) &=& \sum_{a \in (A/\pfrak A)^\times} \chi^{-1}(a)(E_{(0,a)}^{(k)}|_k[\left( \begin{smallmatrix} 0 & -1 \\ 1 & 0  \end{smallmatrix} \right)\left( \begin{smallmatrix} \pfrak & 0 \\ 0 & 1  \end{smallmatrix} \right)])(z) \\
 &=& \sum_{a \in (A/\pfrak A)^\times} \chi^{-1}(a) (E_{(a,0)}^{(k)}|_k[\left( \begin{smallmatrix} \pfrak & 0 \\ 0 & 1  \end{smallmatrix} \right)])(z) \\
 &=& \sum_{a \in (A/\pfrak A)^\times} \chi^{-1}(a) \sum_{c,d \in A} \frac{1}{((a+c\pfrak)\pfrak z + \pfrak d)^k} \\
 &=& \frac{1}{\pfrak^{k}} \sum_{c \in A\setminus \{0\}} \chi^{-1}(c) \sum_{d \in A} \frac{1}{(cz+d)^k} \\
 &=& -\frac{\pitilde^k}{\pfrak^{k}} \sum_{c \in A_+} \chi^{-1}(c) G_k(u(c z)).
\end{eqnarray*}
In the final line we have used that $s_\chi \equiv -k \pmod{q-1}$ and \eqref{Gkeq}, allowing us to collapse the sum over $c$ to the monics. 

\cite[Th. 3.1]{BLadm11} shows that a rigid analytic function with non-vanishing $A$-expansion at infinity is non-zero. Hence, we deduce the non-vanishing of both the forms mentioned above. We conclude that $\frac{\pfrak^{k}}{\pitilde^k}E_{\chi}^{(k)}|_{k}[W_\pfrak]$ is a non-zero element of $M_k^{k}(\pfrak, \chi^{-1})$ with $u$-expansion coefficients in $K(\chi)$.
\end{proof}

We notice here that the forms $E_{\chi}^{(k)}$ have a type of $A$-expansion as well. We shall give further examples with similar expansions in the section on twisting below. 

\begin{proposition} Let $k$ be a positive integer. For all $\chi \in \widehat{(A/ \pfrak A)^\times}$ satisfying \eqref{eisnonvaneq} we have
\begin{equation} \label{Aexp3} 
\frac{\pfrak^k}{\pitilde^k} E_\chi^{(k)}(z) = \sum_{a \in (A/\pfrak A)^\times} \chi^{-1}(a) G_k(u(\frac{a}{\pfrak})) - \sum_{c \in A_+}  \sum_{a \in (A/\pfrak A)^\times} \chi^{-1}(a) G_k(u(cz + \frac{a}{\pfrak})).
\end{equation}

Further, the constant term $\sum_{a \in (A/\pfrak A)^\times} \chi^{-1}(a) G_k(u(\frac{a}{\pfrak}))$ is non-zero and equals the Goss abelian $L$-value $-\frac{\pfrak^k}{\pitilde^k}\sum_{a \in A_+} \frac{\chi^{-1}(a)}{a^k}$.
\end{proposition}
\begin{proof}
An easy computation gives
\begin{equation} \label{Aexp2}
E_{(0,a)}^{(k)}(z) = \frac{\pitilde^k}{\pfrak^k} \sum_{c \in A} G_k(u(cz + \frac{a}{\pfrak})).\end{equation}
Averaging against $\chi^{-1}$ over $a \in (A/\pfrak A)^\times$, we may collapse the sum over non-zero $c$ to the monics to obtain something non-zero because $s_\chi \equiv -k \pmod{q-1}$, giving the claimed identity.

Finally, we demonstrate that the constant term of \eqref{Aexp3} is non-zero by relating it to a special value of Goss' abelian $L$-series. This should be compared with \cite[Prop. 17]{BA-FPjnt}. We have 
\[\sum_{a \in (A/\pfrak A)^\times} \chi^{-1}(a) G_k(u(\frac{a}{\pfrak})) = \frac{1}{\pitilde^k}\sum_{a \in (A/\pfrak A)^\times} \sum_{b \in A} \frac{\chi^{-1}(a)}{(a/\pfrak + b)^k} = -\frac{\pfrak^k}{\pitilde^k}\sum_{a \in A_+} \frac{\chi^{-1}(a)}{a^k}.\]
Again, the collapse down to the monics gives something non-zero because we assume $s_\chi \equiv -k \pmod{q-1}$. 
\end{proof}

\begin{remark}
We use \eqref{linearfrac} to rewrite \eqref{Aexp3} as follows:
\begin{eqnarray*}
\frac{\pfrak^k}{\pitilde^k} E_\chi^{(k)}(z) & = & \sum_{a \in (A/\pfrak A)^\times} \chi^{-1}(a) G_k(u(\frac{a}{\pfrak})) \\ 
 && - \sum_{c \in A_+}  \sum_{a \in (A/\pfrak A)^\times} \chi^{-1}(a) G_k\left(\frac{u(cz)}{\exp_C(\pitilde a / \pfrak)u(cz) + 1}\right). 
 \end{eqnarray*}
From this we obtain further information about the $u$-expansion coefficients of the functions $\frac{\pfrak^k}{\pitilde^k} E_\chi^{(k)}$. Indeed, they lie in the extension of $K$ obtained by adjoining the values of $\chi$ and an element of Carlitz $\pfrak$-torsion, e.g. $\exp_C(\pitilde/\pfrak)$. We shall return to these observations below. 
\end{remark}

\subsection{Linear independence}
Gekeler has computed the number of cusps for $\Gamma_1(\mfrak)$ in \cite[Prop. 6.6]{EGjnt01}, and for the case we are interested in where $\mfrak = \pfrak$, a monic irreducible, this number is $2\frac{|\pfrak|-1}{q-1}$. The next result demonstrates that in each weight $k$ there are exactly as many linearly independent Eisenstein series with character for $\Gamma_0(\pfrak)$ as there are cusps for $\Gamma_1(\pfrak)$.

\begin{proposition} \label{liprop}
Let $k$ be a positive integer. The set 
\[\{E_\chi^{(k)}, E_\chi^{(k)}|_k[W_\pfrak] : s_\chi \equiv -k \pmod{q-1}\}\] 
consists of $2\frac{|\pfrak|-1}{q-1}$ linearly independent Eisenstein series.
\end{proposition}
\begin{proof}
We have shown that all of the forms given in the statement are non-zero and the forms in $\{E_\chi^{(k)} : s_\chi \equiv -k \pmod{q-1}\}$ (resp. in $\{E_\chi^{(k)}|_k[W_\pfrak] : s_\chi \equiv -k \pmod{q-1}\}$) respectively lie in distinct $\chi$-eigenspaces for the action of $\Gamma_0(\pfrak)$. If two forms $E_{\chi_1}^{(k)}$ and $E_{\chi_2}^{(k)}|_k[W_\pfrak]$  lie in the same $\chi$-eigenspace we have that $E_{\chi_1}^{(k)}$ is non-zero at the cusp at infinity while $E_{\chi_2}^{(k)}|_k[W_\pfrak]$ vanishes at the infinite cusp; hence, they are linearly independent. 
\end{proof}

\subsection{Hecke action on Eisenstein series with character}
Let  us give an example of the Hecke action on the Eisenstein series of Proposition \ref{liprop}. 

\begin{proposition} \label{eiseigenprop}
Let $\qfrak$ be a monic irreducible, distinct from the level $\pfrak$, and $a \in (A/\pfrak A)^\times$.
\[T_\qfrak E_{\chi}^{(k)} = \qfrak^k \chi(\qfrak) E_\chi^{(k)}, \text{ and} \]
\[T_\qfrak (E_{\chi}^{(k)}|_{k}[W_\pfrak]) = \qfrak^k E_{\chi}^{(k)}|_{k}[W_\pfrak]. \]
\end{proposition}

\subsubsection{Remarks} 
1. It follows that the naive B\"ockle $L$-function associated to the $E_\chi^{(k)}$ via their eigensystem is a Goss abelian $L$-function. Of course, one wants such $L$-functions to come from cuspidal eigenforms, and we have constructed such forms above. 

2. In each weight, we obtain $\frac{|\pfrak|-1}{q-1}$ linearly independent Eisenstein series $E_{\chi}^{(k)}|_{k}[W_\pfrak]$ for $\Gamma_1(\pfrak)$ all with the same eigensystem outside of the level. This is essentially forced by the existence of $A$-expansions for these forms. 
This result demonstrates a disconnect here between the $A$-expansion coefficients and the Hecke eigenvalues which is in contrast both to the results of Petrov, e.g. \cite[Th. 2.3]{APjnt}, and the classical situation. NB. A similar observation was already made in \cite[Remark 5.25]{Pel-PerkVMF} for the Eisenstein series of Example \ref{dereisserex}. 

3. The remaining $\frac{|\pfrak|-1}{q-1}$ Eisenstein series $E_{\chi}^{(k)}$ have distinct eigensystems, though they do not possess an $A$-expansion in the sense of Petrov. 
Finally, we note that on the $E_{\chi}^{(k)}$, the operator $T_\qfrak$ agrees with $\cdot|_k[\left( \begin{smallmatrix} \mu & \nu \\ \pfrak & \qfrak  \end{smallmatrix} \right)]$ followed by multiplication by $\qfrak^k$, where $\left( \begin{smallmatrix} \mu & \nu \\ \pfrak & \qfrak  \end{smallmatrix} \right) \in \Gamma_0(\pfrak)$.

\begin{proof}[Proof of Proposition \ref{eiseigenprop}] We begin the proof with a couple of lemmas. 
The following result should be well-known. We include it and its proof here for completeness.

\begin{lemma} \label{distrellem}
Let $\qfrak$ be a monic irreducible, distinct from $\pfrak$, and $a \in A$, arbitrary.
\[\sum_{|\beta| < |\qfrak|} G_k(u(c(\frac{z+\beta}{\qfrak}) + \frac{a}{\pfrak})) = \left\{ \begin{array}{ll} \qfrak^k G_k(u(cz +\frac{a\qfrak}{\pfrak})) & \text{ if } (c,\qfrak) = 1, \\ 0 & \text{ if } (c,\qfrak) = \qfrak. \end{array} \right.\]
\end{lemma}
\begin{proof}
This follows just as in the proof of \cite[Th. 2.3]{APjnt}. We have
\begin{eqnarray*}
\sum_{|\beta| < |\qfrak|} G_k(u(c(\frac{z+\beta}{\qfrak}) + \frac{a}{\pfrak})) = \frac{\qfrak^k}{\pitilde^k} \sum_{|\beta| < |\qfrak|} \sum_{d \in A} \frac{1}{(cz + \frac{a\qfrak}{\pfrak} + c\beta + d\qfrak)^k}.
\end{eqnarray*}
As Petrov argues, the map on $\{|\beta|<|\qfrak|\} \times A$ sending $(\beta,d)$ to $c\beta + d\qfrak$ is bijective if $(c,\qfrak) = 1$. Rearranging this absolutely convergent sum gives the first claim. If $(c,\qfrak) = \qfrak$, the vanishing follows from the $A$-periodicity of $u$. 
\end{proof}

\begin{corollary}
Let $\qfrak$ be a monic irreducible, distinct from the level $\pfrak$, and $a \in (A/\pfrak A)^\times$.
\[T_\qfrak E_{(0,a)}^{(k)} = \qfrak^k E_{(0,\qfrak a)}^{(k)}\]
\end{corollary}
\begin{proof}
Using the $A$-expansion \eqref{Aexp2} and the previous lemma, we have
\begin{eqnarray*}
(T_\qfrak E_{(0,a)}^{(k)})(z) &=& \qfrak^{k}(E_{(0,a)}^{(k)}|_k[\left( \begin{smallmatrix} \mu & \nu \\ \pfrak & \qfrak  \end{smallmatrix} \right)\left( \begin{smallmatrix} \qfrak & 0 \\  0 & 1  \end{smallmatrix} \right)](z) + \sum_{|\beta|<|q|} E_{(0,a)}^{(k)}|_k[\left( \begin{smallmatrix} 1 & \beta \\ 0 & \qfrak \end{smallmatrix} \right)](z)) \\
&=&  \qfrak^{k}E_{(0,\qfrak a)}^{(k)}(\qfrak z) +  \frac{\pitilde^k}{\pfrak^k} \sum_{c : (c,\qfrak) = 1} \qfrak ^k G_k(u(cz + \frac{a\qfrak}{\pfrak})) \\
&=& \qfrak^k E_{(0,\qfrak a)}^{(k)}(z).
\end{eqnarray*} \end{proof}

Now we may conclude the proof of Proposition \ref{eiseigenprop}. Immediately from the previous corollary, we have
\[T_\qfrak E_{\chi}^{(k)} = \sum_{a \in (A/\pfrak A)^\times} \chi^{-1}(a) T_\qfrak E_{(0,a)}^{(k)} = \qfrak^k \chi(\qfrak) E_\chi^{(k)}. \]

With Lem. \ref{distrellem} and the $A$-expansion from \eqref{Aexp1}, we obtain
\begin{eqnarray*}
T_\qfrak (E_{\chi}^{(k)}|_{k}[W_\pfrak]) &=& \qfrak^k \chi^{-1}(\qfrak) \sum_{a \in A_+} \chi^{-1}(a) G_k(u(\qfrak a z)) + \qfrak^k \sum_{\substack{a \in A_+ \\ (a,\qfrak) = 1}} \chi^{-1}(a) G_k(u(a z)) \\
&=& \qfrak^k E_{\chi}^{(k)}|_{k}[W_\pfrak].
\end{eqnarray*}
This finishes the proof of Proposition \ref{eiseigenprop}.
\end{proof}

\subsection{Congruences} \label{congrsection}
One should compare the results  of this section with the congruences obtained by Petrov \cite{APjnt} and by C. Vincent \cite{CVjnt} and to those originally found by Gekeler \cite{EGinv}. A similar congruence is shown for Gekeler's false Eisenstein series $E$ in \cite[Theorem 4.13]{Pel-PerkVMF}. 

In this final section, we will write, for general $k$ and $\chi$, 
\[\hat E_\chi^{(k)} := -\frac{\pfrak^k}{\pitilde^k} E_\chi^{(k)}|_k[W_\pfrak] \in M_k^{k}(\pfrak, \chi^{-1}),\] 
which has an $A$-expansion in the sense of Petrov given by \eqref{Aexp1}. 

\begin{proposition} For all $k \geq 1$ and all $\chi$ of prime conductor $\pfrak$, 
\[(\varpi_\chi \hat E_\chi^{(k)})(z) = g(\chi^{-1}) \frac{\pfrak^{k-1}}{\pitilde^k}({E}_\chi^{(k)}(\pfrak z) - {E}_\chi^{(k)}(z)).\]
\end{proposition}
\begin{proof}
 Applying $\varpi_\chi$, we have $\varpi_\chi \hat E_\chi^{(k)} \in M_k(\pfrak^2,\chi)$ and 
\begin{eqnarray} \label{Ehattwist}
(\varpi_\chi \hat E_\chi^{(k)})(z) &=& \frac{g(\chi^{-1})}{\pfrak} \sum_{\substack{a \in A_+ \\ (a,\pfrak) = 1}} \sum_{|\beta| < |\pfrak|} \chi^{-1}(\beta) G_k\left( { u(az + \beta / \pfrak) } \right).
\end{eqnarray}
Upon comparison with \eqref{Aexp3}, we obtain the desired identity.
\end{proof}

To follow, we let $s$ be a non-negative integer. For each monic irreducible $\pfrak$ choose a root $\zeta$ of $\pfrak$ such that $\chi_\zeta$ agrees with the reduction of the Teichm\"uller character, as described in \cite[\S 2.3]{BA-FPjnt}. For $\pfrak$ such that $|\pfrak|>2+s(q-1)$, let \[\chi_{\pfrak,s} := \chi_\zeta^{|\pfrak| - 2 - s(q-1)}.\]
Finally, recall the definition of Petrov's special family $f_s$ from \eqref{specialfam}.

\begin{proposition} \label{SFcongs}
 We have the following congruence of $u$-expansion coefficients:
\[ \hat E_{\chi_{\pfrak,s}}^{(1)} - f_s \in (\theta - \zeta) A[\zeta][[u]].\]
\end{proposition}
\begin{proof}
From \eqref{Aexp1}, we obtain
\[ \hat E_{\chi_{\pfrak,s}}^{(1)}(z) = \sum_{a \in A_+} \chi_\zeta^{1+s(q-1)}(a) u(az) = \sum_{a \in A_+} a(\zeta)^{1+s(q-1)} u(az). \]
Thus the $a$-th $A$-expansion coefficient of the difference $\hat E_\chi^{(1)} - f_s$ is $a(\zeta)^{1+s(q-1)} - a^{1+s(q-1)}$ which is divisible by $\zeta - \theta$. The claim follows after noticing that $u(az)$ has its $u$-expansion coefficients in $A$, for all $a \in A$. 
\end{proof}

\begin{corollary} \label{nonvancongr}
Let $s$, $\pfrak$, $\zeta$ and $\chi_{\pfrak, s}$, as in the previous proposition. We have
\[\frac{g(\chi_{\pfrak,s}^{-1})}{\pitilde} (E_{\chi_{\pfrak,s}}^{(1)}(\pfrak z) - E_{\chi_{\pfrak,s}}^{(1)}(z)) - \varpi_{\chi_{\pfrak,s}} f_s(z) \in (\theta - \zeta) A[\zeta][[u]]. \]
\end{corollary}
\begin{proof}
This follows immediately from the previous two propositions and Proposition \ref{normalizationlem}. 
\end{proof}

\end{document}